\documentclass[11pt]{amsart}
\usepackage{amsmath}
\usepackage{enumerate}
\usepackage{amssymb}
\usepackage{pb-diagram}
\usepackage{enumerate}
\usepackage{color}

\newtheorem{theorem}{Theorem}[section]
\newtheorem{thm}[theorem]{Theorem}
\newtheorem{prop}[theorem]{Proposition}

\newtheorem{fact}[theorem]{Fact}
\newtheorem{cor}[theorem]{Corollary}
\newtheorem{proposition}[theorem]{Proposition}
\newtheorem{lemma}[theorem]{Lemma}

\newtheorem{conj}[theorem]{Conjecture}

\newtheorem*{thm*}{Theorem}

\theoremstyle{definition}
\newtheorem{defn}[theorem]{Definition}

\newtheorem{example}[theorem]{Example}
\newtheorem{notation}[theorem]{Notation}

\newtheorem{qn}[theorem]{Question}

\theoremstyle{remark}

\newtheorem{remark}[theorem]{Remark}

\newcommand{\Z}{\mathbb{Z}}
\newcommand{\N}{\mathbb{N}}

\newcommand{\R}{\mathbb{R}}

\newcommand{\CM}{\mathcal M}
\newcommand{\sub}{\subseteq}
\newcommand{\CN}{\mathcal N}
\newcommand{\ra}{\rangle}
\newcommand{\la}{\langle}

\newcommand{\CH}{\mathcal H}

\newcommand{\C}{\mathbb{C}}

\newcommand{\CG}{\mathcal G}

\usepackage[textsize=tiny]{todonotes}

\newcommand{\CR}{\mathcal R}

\newcommand{\CV}{\mathcal V}

\newcommand{\CU}{\mathcal U}

\newcommand{\cal}[1]{\ensuremath{\mathcal{#1}}}

\title{Locally definable subgroups of semialgebraic groups}

\author {El\'ias Baro}

\address{Departamento de \'Algebra, Geometr\'ia y Topolog\'ia, Facultad de Matem\'aticas, Universidad Complutense de Madrid, Madrid, Spain}

\email{eliasbaro@pdi.ucm.es}

\author {Pantelis  E. Eleftheriou}

\address{Department of Mathematics and Statistics, University of Konstanz, Box 216, 78457 Konstanz, Germany}

\email{panteleimon.eleftheriou@uni-konstanz.de}

\thanks{The first author was supported by the program MTM2017-82105-P. The second author was supported by an Independent Research Grant from the German Research Foundation (DFG) and a Zukunftskolleg Research Fellowship.}

\author{Ya'acov Peterzil}
\address{Department of Mathematics, University of Haifa, Haifa, Israel}
\email{kobi@math.haifa.ac.il}

\begin{document}


\date{\today}
\begin{abstract}
We prove the following instance of a conjecture stated in \cite{EP2}. Let $G$ be an
abelian semialgebraic group over a real closed field $R$ and let $X$ be a semialgebraic subset of $G$. Then the group generated by $X$ contains a generic set and, if connected, it is divisible.

More generally, the same result holds when $X$ is definable in any o-minimal expansion of $R$ which is elementarily equivalent to
 $\mathbb R_{an,exp}$.

We observe that the above statement is equivalent to saying: there exists an $m$ such that $\Sigma_{i=1}^m(X-X)$
 is an approximate subgroup of $G$.

\end{abstract}

 \maketitle

\section{Introduction}

Locally definable groups arise naturally in the study of definable groups in o-minimal structures.
In this paper we are mostly interested in
{\em definably generated}  groups, namely locally definable groups which are
generated by definable sets (see Section \ref{preliminaries}  for basic definitions).
An important example of such groups is the
universal cover of a definable group. Indeed, a definable group in an o-minimal structure can be endowed with a definable
manifold structure making the group into a topological group and then, similarly to the Lie context, one can
construct its universal covering group, in the category of locally definable groups, see \cite{edel2}. This universal covering is generated by a definable set.

The universal covering is an example  of a locally definable group $\CU$
with a definable {\em (left) generic} set $X$; that is, a definable set such that $AX=\CU$ for some countable
 subset $A\sub \CU$ (see \cite[Lemma 1.7]{EP}).
 In \cite{EP2}, it was conjectured that every definably generated abelian group is of this form:
\begin{conj}\label{conj1} Let $\CU$ be an abelian, connected, definably generated group. Then $\CU$ contains a
definable generic set.
\end{conj}

Note that by \cite[Claim 3.11]{EP2}, we may assume in the above conjecture that $\CU$ is generated by a definably compact set.

It has been shown in recent papers that the above conjecture can be restated in several ways (see for example \cite{BEM}).
We will be using the equivalences below, for which we first need a definition.
\begin{defn}
 Given an abelian, connected, definably generated group $\CU$,
we say that a locally definable normal subgroup $\Gamma < \CU$ is a {\em lattice} if $\dim(\Gamma)=0$, and $\CU/\Gamma$
 is definable; that is, there exist a definable group $G$ and a locally definable surjective homomorphism from
 $\CU$ onto $G$, whose kernel is $\Gamma$.
 \end{defn}
\begin{fact}[{\cite[Proposition 3.5]{EP2} and \cite[Theorem 2.1]{EP}}]\label{Thm21}
Let $\CU$ be an abelian, connected, definably generated group.  Then the following
are equivalent:
\begin{itemize}
 \item [(1)] $\CU$ contains a definable generic set.
 \item [(2)] $\CU$ admits a lattice.
 \item [(3)] $\CU$ admits a lattice isomorphic to $\mathbb{Z}^k$, for some $k$.
\end{itemize}
Moreover, each of the above clauses implies that $\CU$ is divisible.
\end{fact}

In this note, we study Conjecture \ref{conj1} for definably generated subgroups
of definable groups.
 To that aim, we introduce the following notion.

\begin{defn}\label{genproperty}Let $\mathcal{M}$ be an o-minimal structure.
 We say that an abelian locally definable group $\CG$ has the  \emph{generic property with respect to
 $\mathcal{M}$} if every definably generated subgroup of $\CG$ contains a definable generic set.
 We omit the reference to $\mathcal{M}$ if it is clear from the context (see Remark
\ref{rmkgenprop}).
\end{defn}

The main result of \cite{EP} can be stated as follows.

\begin{fact}\label{EP}Let $\mathcal{R}$ be an $\aleph_1$-saturated o-minimal expansion of a real closed field $R$.
Then $\langle R^n, + \rangle$ has the generic property with respect to $\cal R$.
\end{fact}

Our first result, in Section \ref{sec:extensions}, is that the generic property can
be lifted under the presence of an exact sequence (Theorem \ref{t:Thmexten}).

\begin{thm*}Let $\mathcal{M}$ be an o-minimal structure. Assume that we are given an exact sequence of abelian locally definable groups and maps:
$$
\begin{diagram}
\node{0}\arrow{e}\node{\mathcal H} \arrow{e,t}{i}\node{\mathcal G}
\arrow{e,t}{\pi}\node{\mathcal V} \arrow{e} \node{0}
\end{diagram}
$$
If $\CV$ and $\CH$ have the generic property, then so does $\CG$.
\end{thm*}

This is a useful criterion that can be applied inductively in certain situations. As
a corollary, we prove (Subsection \ref{subsec:appl}) the following theorem.

\begin{thm*} Let $\CM$ be an $\aleph_1$-saturated o-minimal structure. \begin{enumerate}
\item If $G$ is a definable abelian torsion-free group, then $G$ has the generic
property. \item If $\CM$ expands a real closed field $R$ and $G\sub Gl(n,R)$ is a definable abelian linear group, then $G$
has the generic property.
\end{enumerate}
\end{thm*}

In Section \ref{sec:semigr}, we apply the above lifting result to study definably generated subgroups of semialgebraic
 groups.  In order to formulate the next result, recall that $\mathbb R_{an,exp}$ is
 the expansion of the real field by the real exponential map and all restrictions of real analytic functions
 to the closed unit box in $\mathbb R^n$, for all $n\in \mathbb N$. By \cite{DM}, it is o-minimal.
 The following theorem is the  main result of the paper (Theorem \ref{t:mainsemi}), which generalizes Fact \ref{EP} above.

\begin{thm*}Let $\CR$ be an $\aleph_1$-saturated o-minimal expansion of a real closed field $R$ such that the theory $Th(\mathcal R)\cup Th(\mathbb R_{an,exp})$ is consistent and has an o-minimal completion.
 Then any abelian $R$-semialgebraic group $G$ has the generic property with respect to $\mathcal R$.
In particular, any semialgebraically generated subgroup of
$G$ contains a semialgebraic generic set.
\end{thm*}

A special case of the above result is when $\mathcal R$ is elementarily equivalent to $\R_{an,exp}$.

\medskip

A crucial key case of the above theorem is when $G$ is an  abelian variety. In \cite{ps2} the authors prove the definability in
 $\mathbb{R}_{\text{an,exp}}$, on appropriate domains, of embeddings of families of abelian varieties into projective spaces.
 From those results it is possible to extract the following non-standard property of abelian
 varieties.
\begin{fact}\cite{ps2}\label{AbelianTorus} Let $\mathcal{R}=\langle R, \ldots \rangle$ be a model of $Th(\mathbb R_{\text{an,exp}})$
and let $A\sub \mathbb{P}^N(K)$, $K=R(i)$, be an embedded abelian variety of dimension $g$. Then there exist a locally
 definable subgroup $\CG$ of $\la R^g, +\ra$ and a locally definable covering homomorphism $p:\CG \rightarrow A$.
\end{fact}

For the sake of completeness, we provide a proof of the above fact in Appendix \ref{appendix:abel}. Another important ingredient is the work of E. Barriga on semialgebraic groups, \cite{B}, which we recall in Fact \ref{barriga}.

\subsection{The connection to approximate subgroups}

Approximate subgroups have been studied extensively in various fields including model theory, see for example \cite{Br} and \cite{Hr}.

\begin{defn}
Given a group $G$, and $k\in \mathbb N$,  a set $X\sub G$ is called a {\em $k$-approximate group} if   $X=X^{-1}$ and
there is a finite set $A\sub G$ of cardinality $k$ such that $X\cdot X\sub A\cdot X$. We say that $X$ is an {\em approximate group} if it is $k$-approximate for some $k\in \N$.
\end{defn}

As we observe in Remark \ref{rmkgenprop} below, the existence of a generic set inside a definably generated group $\langle X\rangle\sub G$ is equivalent to saying
that there exists an $m$ such that the set $X(m)$ (the addition of $X-X$ to itself $m$ times) is an approximate group. Thus  our various results and conjectures can be re-formulated in the language of approximate subgroups. For example, Conjecture \ref{conj1} can be re-formulated as follows.

\begin{conj}\label{conj2} Let $\CU$ be a locally definable abelian group in an o-minimal structure and $X\sub \CU$ a definable set.
 Then there exists $m\in \mathbb N$
such that $X(m)$ is an approximate group.
\end{conj}

Our main result above (Theorem \ref{t:mainsemi}) easily implies the following uniformity statement:
\begin{thm} \label{uniformity} Let $\mathcal R=\langle \mathbb R, <,+,\cdot, \dots\rangle$ be an o-minimal expansion of $\mathbb R_{an,exp}$. Let $\{G_t: t\in T\}$ an $\mathbb R_{an,exp}$-definable family of semialgebraic abelian groups, and  $\{X_t:t\in T\}$ an  $\mathcal R$-definable family, with each $X_t\sub G_t$. Then there is  $k\in \mathbb N$,
such that for every $t\in T$,
the set $X_t(k)$ is a $k$-approximate subgroup of $G$.
\end{thm}

In Conjecture \ref{conj2} we restricted our discussion to definable sets in o-minimal structures, but the same problem
could be formulated for arbitrary smooth curves in $\mathbb R^n$.

\begin{qn}\label{question1}
Let $X\sub \mathbb R^n$ be a connected smooth curve. Is there
$m\in \mathbb N$
such that $X(m)$ is an approximate subgroup of $\la \mathbb R^n,+\ra$?
\end{qn}

Let us see that when $X$ is compact the answer to the above question is positive: Indeed, without loss of generality, $0\in X$ and $X$ is given by $\gamma:\R\to \mathbb R^n$. Moreover, we can assume that $\R^n$ is the minimal linear space containing $X$. Thus, there are $t_1,\ldots, t_n$ such that $\dot{\gamma(t_1)},\ldots, \dot{\gamma(t_n)}$ form a basis for $\R^n$ (otherwise, $\R^n$ would not be minimal).
It follows that the map
$$(x_1,\ldots, x_n)\mapsto \gamma(x_1)+\cdots+\gamma(x_n):\R^n\to \R^n$$ is a submersion at the point $(\gamma(t_1),\ldots, \gamma(t_n))$ and hence the point\linebreak $\gamma(t_1)+\cdots +\gamma(t_n)$ is an internal point
of $X(n)$ inside $\R^n$. Since $X(n)+X(n)$ is compact it can be covered by finitely many translates of $X(n)$, so that $X(n)$
is an approximate subgroup.

Note that even if the answer to Question \ref{question1} is positive,  one does not expect any uniformity statement such as that of  Theorem \ref{uniformity} to hold at this level of generality.

We finish this part of the introduction by pointing out that one cannot expect a positive answer to the above question
without the model theoretic (o-minimality) or the topological (smoothness) assumptions. The example was suggested to us by P. Simon. A similar example was also proposed by E. Breuillard.

\begin{example}Let $G=\mathbb R^{\mathbb N}$ with coordinate-wise addition and let $X\sub \R^{\mathbb N}$ be the set of all elements with at most one nonzero coordinate. We claim that for no $n$ is the set $X(n)$ an approximate subgroup.
Indeed,  assume that the set $X(n+1)$ is covered by finitely many translates of $X(n)$.
%
 Let $p: \R^{\mathbb N}\rightarrow  \R^{2(n+1)}$ be the projection onto the first $2(n+1)$ coordinates. The set $p(X(n))$ consists of the tuples with at most $2n$ coordinates different than $0$,  so for any finite subset $A$ of $\R^{\mathbb N}$ we have that $p(A+X(n))=p(A)+p(X(n))$ has dimension $2n$. On the other hand, $p(X(n+1))=\R^{2(n+1)}$, a contradiction.\medskip

Because $\la \R^{\mathbb N},+\ra$ is isomorphic as a group to $\la \mathbb R,+\ra$, we can also find a set $X\sub \mathbb R$ such that for no $n$ is the set $X(n)$  an approximate subgroup.

\end{example}

\subsection{The non-abelian case}It has been shown in \cite{BEM}  that Fact \ref{Thm21} fails for non-abelian groups. More precisely, it was shown
that every definable centerless group, in a sufficiently saturated o-minimal structure, contains a definably generated subgroup with a definable  generic set, which is not the cover of any definable group. However, as far as we know the following question is still open.

\begin{qn} Let $\CU$ be a definably generated group in an o-minimal structure. Does $\CU$ contain a definable generic set?
\end{qn}

In \cite[Section 7]{NIP} there is a discussion of locally definable (called Ind-definable) groups and it is shown (see Proposition 7.8 there) that
every locally definable group contains a definably generated subgroup $\CU$ of the same dimension which contains a definable generic set (using also \cite[Theorem 2.1]{EP}).

\smallskip
\noindent{\emph{Acknowledgements.} We thank Eliana Barriga for reading and commenting on an early version of this paper.

\section{Preliminaries}\label{preliminaries}

 Let $\mathcal{M}$ be an arbitrary $\kappa$-saturated o-minimal structure for $\kappa$
sufficiently large. By \emph{bounded} cardinality, we mean cardinality
smaller than $\kappa$. We refer the reader to \cite{BO} and \cite{ed1} for the basics concerning locally definable groups. A \emph{locally definable group}  is a group $\la \CU,\cdot
\ra$ whose universe is a directed union $\CU=\bigcup_{k\in \mathbb{N}} X_k$ of
definable subsets of $M^n$ for some fixed $n$, and for every $i,j\in \mathbb{N}$,
the restriction of group multiplication to $X_i\times X_j$ is a definable function
(by saturation, its image is contained in some $X_k$).  The dimension of $\CU$ is by
definition $\dim(\CU)=\max\{ \dim(X_k):k\in \mathbb{N}\}$.

A map $\phi:\CU\to \CH$ between locally definable groups is called
{\em locally definable}  if for every definable $X\sub \CU$ and $Y\sub \CH$, the set
$graph(\phi)\cap (X\times Y)$ is definable.
Equivalently, the restriction of $\phi$ to any definable set is a definable map. If $\phi$ is surjective, then there exists a locally definable section $s:\CH \rightarrow \CU$ of $\phi$.

For a locally definable group $\CU$, we say that $\CV\sub \CU$ is {\em a
compatible subset of $\CU$} if for every definable $X\sub \CU$, the intersection
$X\cap \CV$ is a definable set (note that in this case $\CV$ itself is a countable union of definable sets). We say that $\CU$ is \emph{connected} if there is no proper  compatible subgroup of bounded index. By  \cite{ed1}, every locally definable group $\CU$ has a
 connected component $\CU^0$, that is, a connected compatible subgroup of $\CU$ of
 the same dimension. Moreover, $\CU$ admits a locally definable topological structure that makes the group operations continuous. Note that we still use the term ``definably connected'' when referring to
 definable sets. Note also that if $\phi:\CU\to \CV$ is a locally definable
homomorphism between locally definable groups, then $\ker(\phi)$ is a compatible locally definable normal
subgroup of $\CU$. In fact, the following holds.

\begin{fact}{\cite[Theorem 4.2]{ed1}}\label{edmundo} If $\CU$ is a locally definable group and
 $\CH\sub \CU$ is a locally definable normal
subgroup then $\CH$ is a compatible subgroup of $\CU$ if and only if there exists a
locally definable surjective homomorphism of locally definable  groups $\phi:\CU\to
\CV$ whose kernel is $\CH$.
\end{fact}

In Definition \ref{genproperty} we  introduced the notion of an abelian locally
 definable group having the generic property. Now, we stress some easy properties regarding that notion.
 For that, we need the following notation that will be used throughout the paper.

\begin{notation}\label{notation}Let $G$ be an abelian group and $X$  a subset.
The set $X(m)$ denotes the addition of $X-X$ to itself $m$ times. We say that $X$ is symmetric if $X=-X$.
\end{notation}

\begin{remark}\label{rmkgenprop}(1) An abelian locally definable group $\CG$ has the generic property if and
only if for every definable subset $Y\sub \CG$, there are $m,k \in \N$ and $0\in
A\sub Y(3m)$, $|A|\leq k$,  such that $Y(m)+Y(m)\sub A+Y(m)$. In particular, $Y(m)$
is a $k$-approximate group.

\smallskip
\noindent (2) If $\CG$ has the generic property and $\CH$ is a locally definable subgroup of $\CG$, then $\CH$ has also
the generic property.

\smallskip
\noindent (3) Let $\CG$ and $\CV$ be abelian locally definable groups, and
$\pi:\CG\rightarrow \CV$  a surjective locally definable homomorphism. If $\CG$
has the generic property, then so does $\CV$. Indeed, for $X\sub
\CV$ definable, let $Y\sub \CG$ be any definable set with $\pi(Y)=X$ (such $Y$
exists by saturation).
 Since $\CG$ has the generic
property, there are $m,k \in \N$ and a set $0\in A\sub Y(3m)$, $|A|\leq k$ such that
$Y(m)+Y(m)\sub A+Y(m)$. In particular, we get that $0\in \pi(A)\subset X(3m)$ and
$$X(m)+X(m)=\pi(Y(m)+Y(m))\sub \pi(A+Y(m))=\pi(A)+X(m),$$
as required.

\smallskip
\noindent (4) The generic property is preserved under taking reducts. Namely, let $\CM'$ be an o-minimal expansion of $\CM$.
 By (1) above, if $\CG$ is a locally definable group in $\CM$  with the
  generic property with respect to $\CM'$, then $\CG$  has the generic
  property with respect to $\CM$. It is also clear, again using (1), that the generic property is preserved under taking elementary substructures. That is, let $\CN$ be an elementary extension of $\CM$. Let $\CG=\bigcup_{\ell\in \N}X_\ell$ be a locally definable group in $\CM$, and denote by $\CG(N)$ its realization in $\CN$. Then $\CG$ has the generic property with respect to $\CM$ if and only if $\CG(N)$ has the generic property with respect to $\CN$. For, let $Y\subset X_{\ell}(N)$ be a subset of $\mathcal{G}(N)$ definable over a finite tuple $d\in N$. Replace the parameters $d$ by variables $t$, and take the definable set $T= \{t:Y_t\subset X_\ell(N)\}$. Since $T$ is definable over $\CM$, we can consider the definably family in $\CM$ of definable subsets $\{Y_t:t\in T(M)\}$ of $G$. By (1) and saturation of $\CM$ there are $m,k\in \N$ such that for all $t\in T(M)$ there is $0\in
A_t\sub Y_t(3m)$, $|A_t|\leq k$,  such that $Y_t(m)+Y_t(m)\sub A_t+Y_t(m)$, as required.
   \end{remark}

We can now formulate:
\begin{prop} \label{propnew} Let $\CG$ be an abelian locally definable group in $\CM$ (which is still sufficiently saturated). Then the following are
equivalent:
\begin{enumerate}
\item $\CG$ has the generic property. \item For every definable family $\{X_t:t\in
T\}$ of subsets of $\CG$ there exist $m,k\in \mathbb N$ such that for every $t\in
T$, there exists a subset $0\in A\sub X_t(3m)$, of size at most $k$ such that
$$X_t(m)+X_t(m)\sub A+X_t(m).$$
\end{enumerate}

In particular, $\CG$ has the generic property in $\CM$ if and only if it has the
generic property in any/some elementary extension of $\CM$.
\end{prop}
\proof This follows immediately from Remark \ref{rmkgenprop} (1) and saturation.\qed

\begin{remark} While we focus here on o-minimal structures, the notions we defined make sense in any sufficiently saturated structure, in which case Remark
\ref{rmkgenprop} and Proposition \ref{propnew} are also true.
\end{remark}

\section{Group extensions}\label{sec:extensions}

In this section we study the existence of definable generic sets when dealing with abelian group extensions, in an arbitrary
o-minimal structure $\mathcal{M}$. As a corollary, we  prove that definably generated subgroups of abelian torsion-free definable groups contain definable generic sets. When $\CM$ expands a real closed field $R$, we deduce a similar result
for definable linear groups over $R$.

\begin{proposition}\label{propExt}
Assume that we are given an exact sequence of  locally definable abelian groups and maps,
$$
\begin{diagram}
\node{0}\arrow{e}\node{\mathcal H} \arrow{e,t}{i}\node{\mathcal G}
\arrow{e,t}{\pi}\node{\mathcal V} \arrow{e} \node{0}
\end{diagram},
$$
where $\CV$ is connected and admits a lattice. Let $Y\sub \CV$ be a definable generic set and $s:Y\to
\CG$  a definable section. Then the intersection  $\la s(Y)\ra\cap i(\CH)$ is definably
generated.
\end{proposition}

 \proof By  \cite[Lemma 1.7]{EP},
$\mathcal{V}=\la Y \ra$. In particular, $\pi$ sends the group $\la s(Y)\ra$ onto $\CV$. Without loss of generality, we may assume that $i$ is the identity map.

Henceforth we will use that given a definable set $Z\sub \CV$, we can assume that $Z\sub Y$. Indeed, by saturation there
 is $n$ such that $Z\sub Y(n)$ and by definable choice there is a section $r:Z\rightarrow s(Y)(n)\sub \la s(Y) \ra $.
 Thus we can extend the section $s:Y\rightarrow \CG$ to a section $\widetilde{s}:Y\cup (Z\setminus Y)\rightarrow \CG$
 via $r$ in such a way that $\la s(Y)\ra =\la \widetilde{s}(Y\cup Z)\ra$. Therefore we can work with the generic set
 $Y\cup Z$ instead of $Y$, as required. For example, we can assume that $Y\ni 0$ is symmetric (extend the section $s$
  to the set $-Y\cup \{0\}$). Moreover, we can set $s(0)=0$. For, let $y_0:=s(0)\in \CH$ and consider the definable
   section $\widetilde{s}:Y\rightarrow \CG$ such that $\widetilde{s}=s$ in $Y\setminus \{0\}$ and $\widetilde{s}(0)=0$.
    If $\widetilde{D}$ is a definable set which generates $\langle \widetilde{s}(Y)\rangle\cap \CH$, then
     $D:=\widetilde{D}\cup \{y_0\}$ generates  $\langle s(Y)\rangle\cap \CH$, as required.

By Fact \ref{Thm21} and since $Y$ is generic, the locally definable group $\CV$ admits a lattice $\Gamma\simeq \mathbb Z^k$.
Since $\CV/\Gamma$ is definable and $Y$ generic in $\CV$,
 there is a finite set $A\sub \CV$  such that $Y+A+\Gamma=\CV$. Indeed, to see that note that the image
 of $Y$ in $\CV/\Gamma$ is a generic so finitely many translates of it cover the group.
 Now, without loss of generality, we can assume that $A$ contains a fixed set
 of generators $\gamma_1,\ldots,\gamma_k$ of  $\Gamma$. Therefore we can assume that $Y+\Gamma=\CV$ and
 $\gamma_1,\ldots,\gamma_k\in Y$ (extending the section $s$ to $Y+A$).

Let $\Delta=\la s(\gamma_1),\ldots,s(\gamma_k)\ra$ and note that $\pi|_{\Delta}:\Delta\rightarrow \Gamma$
is an isomorphism. Consider the symmetric finite set (notice that $Y(2) \cap \Gamma$ is finite)
$$\Delta_0:=\{ \delta\in \Delta: \pi(\delta)\in Y(2)\}$$
and the definable set
$$D:=(\Delta_0+s(Y)(2))\cap \CH \ \sub \  \la \Delta+ s(Y) \ra \cap \CH \ = \ \la s(Y)\ra \cap \CH.$$
Note that $0\in D$, and we now claim that $D$ generates $\la s(Y)\ra \cap \CH$. To
prove that, it is sufficient to show the following:

\medskip
\noindent {\em Claim.} For all $n$ and for every
$\delta_1,\ldots,\delta_{2^n}\in \Delta$ and $y_1,\ldots, y_{2^n}\in Y$, if
$\Sigma_{i=1}^{2^n} \delta_i+s(y_i)\in \CH$ then $\Sigma_{i=1}^{2^n} \delta_i+s(y_i) \in \la
D\ra$.

\medskip
Indeed, granted the claim, pick $\Sigma^m_{i=1}s(y_i)\in \la s(Y)\ra\cap \CH$. Define $\delta_1=\cdots=\delta_{2^m}=0$ and $y_{m+1}=\cdots=y_{2^m}=0$. Since
$\Sigma_{i=1}^{2^n} \delta_i+s(y_i)=\Sigma^m_{i=1}s(y_i)\in \CH$ we deduce
$\Sigma^m_{i=1}s(y_i)\in \la D \ra$, as required.

\noindent {\em Proof of the claim}. By induction on $n$. The case $n=0$ gives
$\pi(\delta_1)+y_1=0$, hence $\pi(\delta_1)\in Y \sub Y(2)$, so $\delta_1\in
\Delta_0$. Therefore $\delta_1+s(y_1)\in D$.

Assume now that $\Sigma_{i=1}^{2^n} \delta_i+s(y_i)\in \CH$. We want to show that
$\Sigma_{i=1}^{2^n}\delta_i+s(y_i)$ is in $\la D\ra$. We write the sum in pairs:
$$\Sigma_{i=1}^{2^n}(\delta_i+s(y_i))=\Sigma_{k=1}^{2^{n-1}}
(s(y_{2k-1})+s(y_{2k})+\delta_{2k-1}+\delta_{2k}).$$ Now, because $Y+\Gamma=\CV$,
for each $k=1,\ldots, 2^{n-1}$ there is $w_k\in Y$ and
$\beta_k\in \Gamma$ such that $y_{2k-1}+y_{2k}=\beta_k+w_k$. Let $\alpha_k\in
\Delta$ be such that $\pi(\alpha_k)=\beta_k$. Note that $\beta_k\in Y(2)$, so that $\alpha_k \in \Delta_0$. Hence,
$$(s(y_{2k-1})+s(y_{2k})-\alpha_k-s(w_k))\in D.$$
Also because the image under $\pi$ of $s(y_{2k-1})+s(y_{2k})-\alpha_k-s(w_k)$ is
$0$, it belongs to $\CH$.

Thus the above sum also equals
\begin{equation*}
\begin{split}
\Sigma_{i=1}^{2^n}(\delta_i+s(y_i)) = & \Sigma_{k=1}^{2^{n-1}} \big(s(y_{2k-1})+s(y_{2k})-\alpha_k-s(w_k)\big)+ \\
& \Sigma_{k=1}^{2^{n-1}}
\big(\delta_{2k-1}+\delta_{2k}+\alpha_k +s(w_k)\big).
\end{split}
\end{equation*}
We already showed that $\Sigma_{k=1}^{2^{n-1}} \big(s(y_{2k-1})+s(y_{2k})-
\alpha_k-s(w_k)\big)\in \la D\cap \CH\ra$, so if we denote
$\widetilde{\delta}_k:=\delta_{2k-1}+\delta_{2k}+\alpha_k\in \Delta$ then
$$\Sigma_{k=1}^{2^{n-1}}
\big(\widetilde{\delta}_{k}+s(w_k)\big)\in \CH,$$ and it remains to see that it
belongs to $\la D \ra$. This follows by induction, so the claim is proved and with
it Proposition \ref{propExt}.\qed


\begin{prop}\label{p:thmExt}With $\CH$, $\CG$ and $\CV$ as in Proposition \ref{propExt}, assume that $X\sub \CG$ is a
definable set with $\la \pi(X)\ra=\CV$. Then $\la X\ra \cap i(\CH)$ is definably
generated.
\end{prop}
\proof Again, we may assume that $i$ is the identity map. Since $\CV$ admits a lattice it contains a definable generic set $Y$. Without loss of generality, we may assume that $\pi(X)(1)\sub Y$. By saturation $Y\sub \pi(X)(\ell)$ for some $\ell\in \N$ and therefore by definable choice we can pick a section $s:Y\rightarrow \la X \ra$. Moreover, we can assume that $s(\pi(X))\sub X$. Let $E:= X(1)\cap \CH$. By Proposition \ref{propExt} we have that $\CH_0:=\la s(Y)\ra \cap \CH$ is definably generated.
Thus, to prove that $\la X\ra \cap \CH$ is definably generated it suffices to show that $\la X\ra \cap \CH=\la E\ra +\CH_0$.

To that aim, pick $x_1\ldots, x_{n}\in X$ such that $\Sigma_{i=1}^{n} x_i\in \CH$. We can write
$$\Sigma_{i=1}^{n} x_i=\Sigma_{i=1}^{n} (x_i-s(\pi(x_i))+\Sigma_{i=1}^n s(\pi(x_i)).$$
Note that $x_i-s(\pi(x_i))\in E$ for each $i=1,\ldots,n$ and therefore $\Sigma_{i=1}^{n} (x_i-s(\pi(x_i))\in \langle E \rangle$. Moreover,
$\Sigma_{i=1}^n s(\pi(x_i))=\Sigma_{i=1}^{n} x_i-\Sigma_{i=1}^{n} (x_i-s(\pi(x_i)) \in \CH$.
Since also $s(\pi(x_i))\in s(\pi(X))\sub s(Y)$ for each $i=1,\ldots,n$, we get $\Sigma_{i=1}^n s(\pi(x_i))\in \CH_0$ and so
 $\Sigma_{i=1}^{n} x_i\in \la E\ra +\CH_0$, as required.
\qed

\medskip
Before the main corollary we need also the following lemma.
\begin{lemma}\label{lemConn} Let  $\CG$ be an abelian locally definable group which is definably generated. Then its connected component
 is definably generated
by a definably connected set (with regard to the group topology).  In particular, if every {\em connected} definably
generated subgroup of $\CG$ contains a definable generic set, then $\CG$ has  the
generic property.
\end{lemma}
\proof Let $\CG$ be a locally definable group and $X\sub \CG$ be a definable set
which generates $\CG$. Let $X_1,\ldots,X_k$ be its connected components. Fix an
element $a_i$ in each $X_i$, and let $\Gamma=\la a_1,\ldots, a_k\ra$. Consider the
connected set $\widetilde{X}=\bigcup X_i-a_i$, and notice that $\la X\ra= \la
\widetilde{X}\ra+\Gamma$. Since $\la \widetilde{X}\ra$ is a locally definable
subgroup of $\CG$ of bounded index, it is compatible (\cite[Fact 2.3(2)]{EP2}), and since it is connected, it must be the connected component of $\CG$.

For the second part of the statement, let $\CH$ be a definably generated subgroup of
$\CG$. Then, by what we just showed, its connected component $\CH^0$ is definably
generated and therefore by hypothesis it contains a definable generic set $Y$, that
is, there is a bounded $A\sub \CH^0$
 such that $A+Y=\CH^0$. Since  $\CH^0$ has bounded index in $\CH$, there is a bounded $B\sub \CH$
  such that $B+\CH^0=\CH$. In particular $A+B+Y=\CH$, as required. \qed

\begin{thm}\label{t:Thmexten} Assume that we are given an exact sequence of abelian locally definable groups and maps
$$
\begin{diagram}
\node{0}\arrow{e}\node{\mathcal H} \arrow{e,t}{i}\node{\mathcal G}
\arrow{e,t}{\pi}\node{\mathcal V} \arrow{e} \node{0.}
\end{diagram}
$$
If $\CV$ and $\CH$ have the generic property then $\CG$ has the generic property.
\end{thm}
\begin{proof} By Lemma \ref{lemConn} it is sufficient to consider subgroups of $\CG$
which are generated by definably connected sets. Let $X$ be a definably connected
set. Since $\pi(\la X \ra)=\la \pi(X)\ra$ is a definably generated connected group,
we have the exact sequence of locally definable groups
$$0\rightarrow\la X \ra \cap i(\CH) \rightarrow \la X \ra \rightarrow \la \pi(X) \ra \rightarrow 0.$$
By hypothesis the connected group $\la \pi(X)\ra$ contains a definable generic set, that is, there exists a definable
set $Z_1\sub \la X \ra$ such that $\pi(Z_1)$ is generic in $\la \pi(X)\ra$. In particular the group $\la \pi(X)\ra$ admits
 a lattice (see Fact \ref{Thm21}) and therefore by Proposition \ref{p:thmExt} the group $\la X \ra \cap i(\CH)$ is definably
  generated. Again by hypothesis, we have that $\la X \ra \cap i(\CH)$ contains a definable generic
   set $Z_2$. Finally, it is not hard to see that $Z_1+Z_2$ is generic in $\la X\ra.$\end{proof}

\subsection{Some applications of Theorem \ref{t:Thmexten}}\label{subsec:appl}
 First, we   study definably
generated subgroups of abelian torsion-free definable groups (see basic facts on
torsion-free groups definable in o-minimal structures in Section 2.1
 in \cite{ps}).

\begin{cor}\label{torsionfree} Any  abelian torsion-free definable group $G$ in an o-minimal
structure $\CM$  has the generic property with respect to  $\CM$.
\end{cor}
\proof We prove it by induction on $\dim(G)$. Assume first that $\dim(G)=1$ and
prove first a more general result.

\begin{lemma}\label{lem:one}  If $\CU$ is a $1$-dimensional torsion-free
locally definable group then it has the generic property. \end{lemma}
 \proof By \cite[Corollary 8.3]{ed1}, the
group $\CU$ can be linearly ordered. By Lemma \ref{lemConn} it suffices to study a
subgroup generated by a set of the form $(-b,b):=\{x\in G:-b<x<b\}$, that is,
$$\la (-b,b)\ra= \bigcup_{n\in \mathbb{\N}} (-nb,nb).$$

It is easy to verify that the group $\Gamma=\mathbb{Z}b$ is a lattice in $\la
(-b,b)\ra$ because $\la (-b,b)\ra /\Gamma$ is isomorphic to the definable group
$\big([0,b), \text{mod }b \big)$.\qed

Now, assume that $\dim(G)>1$. Then, by \cite{PSe}, there exists a definable subgroup $H$ of $G$ of dimension $1$.
 In particular, we have the exact sequence
$$0\rightarrow H \rightarrow G \rightarrow G/H \rightarrow 0.$$
Since $H$ and $G/H$ are abelian torsion-free definable groups smaller dimension it
follows by induction that they have the generic property. By Theorem
\ref{t:Thmexten}, so does $G$.\qed

Next we prove:
\begin{cor}\label{cor:linear} Let $\CM$ be an o-minimal expansion of a real closed
field and $G\sub Gl(n,R)$ a definable abelian linear group. Then $G$ has the generic
property with respect to $\CM$.
\end{cor}
\proof We may assume that $G$ is definably connected with regard its group topology. By \cite[Proposition
3.10]{PPS}, $G$ is definably isomorphic to a semialgebraic linear group, and hence
it is the connected component of $H(R)$ for some abelian linear algebraic subgroup of
$GL(n,K)$, defined over $R$ (here $K$ is the algebraic closure of $R$). By
\cite[Fact 3.1]{PPS}, $G$ is semialgebraically isomorphic to a group of the form
$T^m\times (R^*_{>0})^k \times (R^+)^n$, where $T=SO(2,R)$.

By Corollary \ref{torsionfree}, the group $(R^*_{>0})^k\times (R^+)^n$ has the
generic property, so by Theorem \ref{p:thmExt} it is enough to show that $T=SO(2,R)$
has the generic property. The universal covering of $T$ is a torsion-free $1$-dimensional locally definable group, so by Lemma \ref{lem:one}, it has the
generic property. Thus $G$ has the generic property.\qed

\section{Semialgebraic groups}\label{sec:semigr}

The main purpose of this section is to show, see Theorem \ref{t:mainsemi} below,
that every semialgebraic abelian group over a real closed field $R$ has  the generic
property with respect to certain o-minimal expansions of $R$, which we now fix.

\medskip
\textbf{ In the rest of the section, we fix $\mathcal{R}$ to  be an $\aleph_1$-saturated o-minimal
$L$-structure expanding a real closed field $R$ such that the $(L\cup
L_{\text{an,exp}})$-theory $\text{Th}(\mathcal{R})\cup T_{\text{an,exp}}$ is
consistent and has an o-minimal completion, call it $T_0$. We denote by $K:=R(i)$
its algebraic closure.}\\

For example, any real closed field, or more generally an $\aleph_1$-saturated  structure elementarily
equivalent to $\mathbb R_{an,exp}$ clearly satisfies the above.

\medskip
We start by analysing the case of abelian varieties.

\begin{prop}\label{abelianGP}Every embedded abelian variety  $A\sub \mathbb{P}^N(K)$ over $K$
has the generic property with
respect to $\mathcal{R}$.
\end{prop}

\begin{proof}First, note that by our assumptions there exists an elementary extension $\CR \prec \CR'$
such that $\CR'$ can be expanded to a model of $T_0$. Furthermore, we may assume
that this structure is $\aleph_1$-saturated. By Proposition \ref{propnew} and the
fact that the generic property is preserved under taking reducts and elementary substructures (Remark \ref{rmkgenprop}(4)),  it is sufficient to prove
the result in $\CR'$. Thus, all in all, we can assume that
 $\mathcal{R}$ is an $\aleph_1$-saturated model of $T_0$.

By Fact \ref{AbelianTorus}, there exist a locally definable subgroup $\CG$ of some $\la R^g, +\ra$
and a locally definable covering homomorphism $p:\CG\rightarrow A$. Thus, by Fact
\ref{EP} and Remark \ref{rmkgenprop}, the group $A$ has the generic
property.\end{proof}

\begin{prop}\label{l:linear} Let $H$ be an irreducible abelian $K$-algebraic group.
Then $H$ has the generic property with respect to $\mathcal{R}$.
\end{prop}

\begin{proof} As in the proof of Proposition \ref{abelianGP}, we can assume that $\mathcal{R}$ is
an $\aleph_1$-saturated model of $T_0:=\text{Th}(\mathcal{R})\cup T_{\text{an,exp}}$.

By Corollary \ref{cor:linear}, the result is true when $H$ is linear (notice that
every linear subgroup of $GL(n,K)$ can be viewed as a linear subgroup of $GL(m,R)$
for some $R$).

For the general case, by Chevalley's theorem, there are a linear group $L$ and an
abelian variety $A$ such that the following is an exact sequence:
$$0 \rightarrow L\rightarrow H \rightarrow  A \rightarrow 0.$$
Thus, by  Theorem \ref{t:Thmexten} and Proposition \ref{abelianGP}, the group $H$
also has the generic property.
\end{proof}
\begin{remark}\label{realALG}If $H$ is an abelian $K$-algebraic  group defined over $R$,
then by Remark \ref{rmkgenprop} and Proposition \ref{l:linear}, the group of $R$-rational points
 $H(R)^\text{o}$ has the generic property.
\end{remark}

Before reaching our main theorem we recall the following result of Barriga,
\cite[Theorem\,10.2]{B}, which describes every semialgebraic group in terms of the
$R$-points of an associated algebraic group over $R$.
\begin{fact}\label{barriga}
Let $G$ be a definably compact and connected semialgebraic abelian group over $R$. Then there
exists a $K$-algebraic group $H$ defined over $R$,
  an open connected locally semialgebraic subgroup $\mathcal{W}$ of the o-minimal universal
   covering group
$\widetilde{H(R)^{\text{o}}}$ of the connected component of $H(R)$, and a locally
semialgebraic surjective covering homomorphism $\theta: \mathcal{W}\rightarrow G$,
with $0$-dimensional kernel.
\end{fact}

\begin{theorem}\label{t:mainsemi}For $\CR$ an o-minimal structure expanding a real closed field $R$,
as before,  let $G$ be an abelian semialgebraic group over $R$. Then $G$ has the
generic property with respect to $\mathcal{R}$. In particular, any semialgebraically
generated subgroup of $G$ contains a generic semialgebraic subset.
\end{theorem}
\begin{proof}By Lemma \ref{lemConn} it is enough to show that every locally
definable subgroup of $G$ generated by a definably connected  set contains a
definable generic subset. Thus, we can assume that $G$ is connected.

By \cite[Theorem 1.2]{PSe}, applied finitely many times, $G$ contains a torsion-free subgroup $H$ such that the quotient $G/H$  is definably compact. Thus, by Theorem \ref{t:Thmexten}
 and Corollary \ref{torsionfree}, we may assume that $G$ is definably compact.

Using the notation of  Fact \ref{barriga}, we have a covering homomorphism
$\theta: \mathcal{W}\rightarrow G$, with $\mathcal W$ a definably generated subgroup
of the locally definable group $\widetilde{H(R)^{\text{o}}}$.

Denote by $p:\widetilde{H(R)^{\text{o}}}\rightarrow H(R)^{\text{o}}$ the universal covering map.
 We have the exact sequence
$$0 \rightarrow \ker(p)\rightarrow \widetilde{H(R)^{\text{o}}} \rightarrow  H(R)^{\text{o}}\rightarrow 0.$$
Note that $\ker(p)$ is discrete and therefore its only semialgebraically generated
connected subgroup is the trivial one, so by Lemma \ref{lemConn} the group $\ker(p)$
has the generic property. Thus, by Theorem \ref{t:Thmexten} and Remark
\ref{realALG}, we deduce that $\widetilde{H(R)^\text{o}}$ has
 the generic property. In particular, by Remark \ref{rmkgenprop}(2), the same is true for $\mathcal{W}$, and by (3), also for $G$, as required.
\end{proof}

\section{Appendix: Abelian Varieties}\label{appendix:abel}

Fact \ref{AbelianTorus} is a consequence of the results in \cite{PSe}. Maybe not in this form, we believe it
is well-known by the experts (e.g.,  a similar statement is used in \cite[\S 5.2.2 and \S 5.3]{S}). For the sake of completeness, we  sketch a proof in this appendix. As in \cite{PSe}, we  quote several facts concerning abelian varieties, see  \cite{BL} for details.

For a positive $g\in \N$, by a \emph{complex $g$-torus} we mean the quotient group $\C^g/\Lambda$ where $\Lambda$ is a lattice, i.e., a subgroup of $(\C^g,+)$ generated by $2g$ vectors which are $\R$-linearly independent. It is  a  compact  complex  Lie group  of  dimension $g$.  A torus $\C^g/\Lambda$ is called
an \emph{abelian variety} if it is biholomorphic to a complex \emph{embedded abelian variety}, namely a projective connected complex algebraic group.

By a theorem of Baily \cite{Bai}, for any $g\in \N$ there is a countable collection of constructible families of embedded abelian varieties $\{\mathcal{A}^D\}_{D\in \mathcal{D}}$, parameterized by certain polarizations $D\in \mathcal{D}$, each of the form
$$\mathcal{A}^D=\{A^D_t\subset \mathbb{P}^N_D(\C): t\in S_D\},$$
such that every $g$-dimensional embedded abelian variety
is isomorphic to a member of  one of the $\mathcal A^D$'s (see also \cite[Thm. 8.11]{ps2}).

We now denote by $\mathbb T$ the semialgebraic compact group $([0,1), +\, \mbox{mod} 1)$, which is isomorphic to $\mathbb R/\mathbb Z$. A consequence of  \cite[Theorem 8.10]{ps2}  is the following:
\begin{fact}\label{theta}For any $g\in \N$  and for each $D\in \mathcal{D}$
there exists in $\mathbb R_{an,exp}$ a definable family of group isomorphisms between the members of $\mathcal A^D$ and
$\mathbb T^{2g}$.
\end{fact}
Indeed, the family of maps which is given by $\Phi^D$ in \cite[Theorem 8.10]{ps2} yields group biholomorphisms between   complex tori of the form $\mathcal E_t=\mathbb C^g/\Lambda_t$ and  the members of  $A_t\in \mathcal A^D$.
Each $\mathcal E_t$, via its fundamental domain, is definably group-isomorphic to $\mathbb T^{2g}$.

\begin{proposition}\label{p:finPol} Let $\mathcal{A}=\{A_t: t\in T\}$ be a constructible family without parameters of  embedded $g$-dimensional abelian varieties of $\mathbb{P}^N(\C)$.

Then there are constructible $T_1,\ldots, T_k\sub T$ covering  $T$, and there are finitely many  $D_1,\ldots,D_k\in \mathcal{D}$ and $d\in \mathbb{N}$ such that for any $j=1,\ldots, k$, and $t\in T_j$ there exists an algebraic isomorphism between $A_t$ and an abelian variety in $\mathcal{A}^{D_j}$. Moreover, for each $j$, there is a constructible family of isomorphisms $\mathcal H^{D_j}$, between the members of  $\{A_t:t\in T_j\}$  and of  $\mathcal A^{D_j}$.
\end{proposition}
\begin{proof} By Fact \ref{theta}, each $A_t\in \mathcal A$ is bi-regularly isomorphic to some $A\in \mathcal{A}^D$, for $D\in \mathcal D$. Because the complex field is $\aleph_1$-saturated and there are countably many $D\in \mathcal D$ it follows
that there exist $D_1,\ldots,D_k\in \mathcal{D}$ such that for any $t\in T$ we
have that $A_t$ is bi-regularly isomorphic to an abelian variety in
$\mathcal{A}^{D_j}$ for some $j\in \{1,\ldots,k\}$. Again by saturation, the
degree of the isomorphism, as $t$ varies in $T_j$, is uniformly bounded by some $d$.

Now, for each $j\in \{1,\ldots,k\}$ the set $T_j$ of $t\in T$ such that there exists a bi-regular isomorphism of degree at most $d$ between $A_t$ and an abelian variety in $\mathcal{A}^{D_j}$ is constructible without parameters. Because of the above bound on the degree,  there exists a constructible family of isomorphisms $\mathcal H^{D_j}$ as required.
\end{proof}

\begin{proof}[Proof of Fact \ref{AbelianTorus}] We now return to the setting of \ref{AbelianTorus} with $R$ a real closed field and $K=R(i)$ its algebraic closure. Let $A\sub \mathbb P^N(K)$ be an embedded abelian variety. Let $c\in K^\ell$ be a tuple of coefficients defining algebraically the variety $A$.


We can replace the parameter $c$ by a tuple $t$ of free variables and therefore we obtain (without parameters) a constructible family $\mathcal{A}=\{A_t:t\in T\}$ of abelian subvarieties of $\mathbb{P}^N(K)$ of dimension $g:=\dim(A)$.

Consider the realization $\mathcal{A}(\C)$ of $\mathcal{A}$ in $\C$. By Proposition \ref{p:finPol} we can assume that $\mathcal A$ is a sub-family of $\mathcal A^D$ for some $D\in \mathcal D$.

By Fact \ref{theta} there is a definable family in $\mathbb{R}_{an,exp}$ of group isomorphisms
between the members of $\mathcal A^D$ and $\mathbb T^{2g}$. Going back to $R$ and $K$, we can find a definable isomorphism between $A$ and $\mathbb T^{2g}$ with its natural realization in $R$. Finally, there is a locally definable covering map from a locally definable subgroup of $R^{2g}$ onto $\mathbb T^{2g}$, so also onto $A$.\end{proof}



\end{document}